\documentclass[11pt]{article}
\usepackage{authblk}
\usepackage{hyperref}
\usepackage{float}

\usepackage{amssymb,amsthm,amsmath,mathrsfs,fullpage,multirow}
\usepackage{tikz}
\usetikzlibrary{matrix}

\makeatletter
\def\keywords{\xdef\@thefnmark{}\@footnotetext}
\makeatother

\newcommand{\ol}[1]{\overline{#1}}

\renewcommand\S{\mathcal S}
\newcommand\A{\mathcal A}

\newtheorem{theorem}{Theorem}[section]
\newtheorem{lemma}[theorem]{Lemma}
\newtheorem{proposition}[theorem]{Proposition}

\newtheorem{conjecture}[theorem]{Conjecture}
\theoremstyle{definition}

\newtheorem{remark}[theorem]{Remark}
\newtheorem{example}[theorem]{Example}

\usepackage{amsmath}
\newcommand{\stirling}[2]{\genfrac{\{}{\}}{0pt}{}{#1}{#2}}

\newcommand{\thistheoremname}{}
\newtheorem*{genericthm*}{\thistheoremname}
\newenvironment{namedthm*}[1]
  {\renewcommand{\thistheoremname}{#1}%
   \begin{genericthm*}}
  {\end{genericthm*}}

\usepackage{enumerate}

\title{Arrow pattern avoidance in permutations: structure and enumeration}
\author[1]{Kassie Archer}
\author[1]{Robert P. Laudone}
\affil[1]{{\small Department of Mathematics, United States Naval Academy, Annapolis, MD 21402}}
\affil[ ]{{\small Email: karcher@usna.edu, laudone@usna.edu}}
\date{}

\begin{document}


\maketitle

\begin{abstract}
    Arrow patterns were introduced by Berman and Tenner as a generalization of vincular patterns. They observed that arrow patterns have the potential to bridge the divide between a permutation's cycle notation and its one-line notation; in support of this, they used arrow avoidance to enumerate shallow and cyclic shallow permutations. More recently, $321$-avoiding cyclic permutations were recharacterized entirely in terms of arrow avoidance. Motivated by these results, we initiate a systematic study of arrow avoidance. In this paper, we prove structural results about arrow patterns, including defining arrow-Wilf equivalence, and enumerate several arrow avoidance classes. Finally, we consider the avoidance of pairs of arrow patterns, focusing on cases that prohibit fixed points in the underlying permutation.
\end{abstract}

\section{Introduction}

Arrow patterns were introduced in \cite{BT22} as a generalization of vincular patterns, which the authors used to characterize shallow permutations in terms of pattern avoidance. At the time, they observed that arrow patterns have the potential to bridge the divide between a permutation's cycle notation and its one-line notation. This divide has increasingly become the focus of recent research. A central problem in this area is the enumeration of cyclic permutations whose one-line notation avoids specific patterns. For example, the authors of \cite{AE14, BC19} enumerated all cyclic permutations avoiding a pair of permutations, except for the pair $(132,213)$. For this remaining case, bounds were established in \cite{H19}. Despite this progress, the problem of enumerating cyclic permutations whose one-line notation avoids any single pattern of size three remains open.

In support of the claim in \cite{BT22}, the second author of the present paper recently proved a new characterization of $321$-avoiding cyclic permutations in terms of arrow pattern avoidance \cite{La26}. This removes the need to treat pattern avoidance and cyclicity as separate conditions, reframing the problem purely in terms of avoidance. Given the results of \cite{BT22, La26}, it is highly plausible that other sets of permutations in which both the one-line and algebraic structure of the permutation are considered can be described in a similar fashion.

Motivated by this potential, we initiate a systematic study of arrow avoidance. We begin in Section \ref{sec:arrows} by proving structural results about arrow patterns, defining arrow-Wilf equivalence, and connecting specific arrow patterns to vincular patterns. In Sections \ref{sec:12tbc}, \ref{sec:21tbc}, and \ref{sec:13tbc}, we enumerate all but two of the arrow avoidance classes for arrow patterns of size 3 whose underlying string has size at most $2$ and that contain a single arrow. In Section \ref{sec:pairs}, we consider the avoidance of two arrow patterns, one of which corresponds to prohibiting fixed points in the underlying permutation. Along the way, several well-known sequences emerge, which we summarize in Table \ref{tab: sequences}. We conclude in Section \ref{sec:conjs} with some open questions and conjectures.

\subsection{Definitions and notation}
Let $\S_n$ denote the set of permutations of $[n]=\{1,2,\ldots, n\}$ for $n\geq 1.$ A permutation $\pi\in\S_n$ can be written in its one-line form as $\pi = \pi_1\pi_2\ldots \pi_n$ where $\pi_i:=\pi(i)$. Alternatively, it can be written in its standard cycle form as a product of disjoint cycles where each cycle starts with its largest element and the cycles are ordered by this largest element. For example, $\sigma = 413526987$ in one-line notation corresponds to $\sigma = (3)(5214)(6)(8)(97)$ in its standard cycle form.

The \emph{fundamental bijection} $\theta:\mathcal{S}_n\to\mathcal{S}_n$ (described in \cite[Page 30]{S11}) uses both the one-line and standard cycle notation of a permutation.
For a permutation $\sigma \in \mathcal{S}_n$, we obtain $\theta(\sigma)$ by writing $\sigma$ in its standard cycle notation and removing the parentheses to obtain a new permutation. For example, considering the same example as in the previous paragraph, $\theta(413526987) = 352146897.$ Since this map is clearly invertible, it is indeed a bijection as its name suggests. In this paper, we will often be concerned with a permutation's inverse image under this bijection, so for ease of notation, we let $\hat\pi=\theta^{-1}(\pi).$ In our example, if $\pi = 352146897$, then $\hat\pi = 413526987.$

We say that a list of positive integers $a_1a_2\ldots a_k$ is order-isomorphic to a permutation $\tau\in\S_k$ if for all $r,s$, we have $a_{r} < a_{s}$ if and only if $\tau_r < \tau_s$. In this case, we define the reduction ${\rm red}(a_1\ldots a_k) = \tau.$
We say that a permutation $\pi = \pi_1\cdots\pi_n$, written in one-line notation, \emph{contains} a pattern $\tau = \tau_1\tau_2\ldots \tau_k$ if there are indices $i_1 < i_2 < \cdots < i_k$ with ${\rm red}(\pi_{i_1}\ldots \pi_{i_k}) = \tau$ (that is, a subsequence of $\pi$ is order-isomorphic to $\tau$). If $\pi$ does not contain $\tau$, we say $\pi$ \emph{avoids} $\tau$. We denote the set of permutations of size $n$ avoiding a set of permutations $\{\tau_1,\dots,\tau_\ell\}$ by $\S_n(\tau_1,\dots,\tau_\ell)$. Two sets of patterns $S$ and $T$ are called \textit{Wilf-equivalent} if $|\S_n(S)|=|\S_n(T)|$. For example, it is well-known that when $\tau\in\S_3$, we have $|\S_n(\tau)|=C_n$, the $n$-th Catalan number, and thus all $\tau\in\S_3$ belong to the same Wilf-equivalence class.

This notion of pattern containment and avoidance can be generalized. A \textit{vincular} pattern is a permutation $\nu$ in which certain consecutive entries are marked. We use an overline for these marked entries. A permutation $\pi \in \S_n$ contains a vincular pattern if it contains a sub-permutation in the same relative order as $\nu$ where entries in the sub-permutation corresponding to marked entries in $\nu$ must occur consecutively in $\pi$. For example, consider the vincular pattern $\nu = 1 \ol{32}$. Then, $\pi = 35241$ avoids $\nu$ because the only occurrence of a $132$ pattern in $\pi$ is $\pi_1\pi_2\pi_4 = 354$, but the entries $\pi_2=5$ and $\pi_4=4$ do not occur consecutively. The following lemma gives the enumerations of $\S_n(\nu)$ when $\nu$ is size 3 with two marked entries.
\begin{lemma}[\cite{Cl01}] \label{lem: vincAvoidance}
   For $n \geq 1$,
    \[
    |\S_n(\nu)| = 
    \begin{cases}
        B_n &\text{if} \;\; \nu\in \{1\ol{23}, 3\ol{21}, \ol{12}3, \ol{32}1, 1\ol{32}, 3\ol{12}, \ol{21}3, \ol{23}1\}\\
        C_n &\text{if} \;\; \nu \in \{2\ol{13}, 2\ol{31}, \ol{13}2, \ol{31}2\}
    \end{cases}
    \]
\end{lemma}

Given $\pi \in \S_n$ with $\pi = \pi_1\ldots\pi_n$, the \emph{reverse} of $\pi$, denoted $\pi^r$, is $\pi_{n} \pi_{n-1} \ldots \pi_2 \pi_1$. The \emph{complement} of $\pi$, denoted $\pi^c$, is $(n-\pi_1 + 1)(n-\pi_2 + 1) \cdots (n-\pi_{n-1} + 1) (n-\pi_n + 1)$. We note that these two operations commute, so $\pi^{rc} = \pi^{cr}$. For example, if $\pi = 623154$, then $\pi^r = 451326$, $\pi^c = 154623,$ and $\pi^{rc} = 326451$. Suppose $S\subseteq [n]$ is a subset of the entries of $\pi$. Then the complement of $\pi$ relative to $S$, denoted $c_S(\pi)$, is the permutation obtained from $\pi$ by replacing the $j$-th smallest entry of $S^c = [n] \setminus S$, with the $j$-th largest entry of $S^c$. So we fix the entries of $S$ and complement the remaining entries. We note that $c_{\varnothing}(\pi) = \pi^c$, the complement in the usual sense, and $c_{[n]}(\pi) = \pi$.

For example, if $\pi = 6{\mathbf{23}}1{\mathbf{54}}$ and $S = \{1,6\}$ so that $S^c = \{2,3,4,5\}$, then $c_S(\pi) = 6{\mathbf{54}}1{\mathbf{23}}$. An example we will use later in this paper is $c_{\{1\}}(\pi)$, which we will denote by $c_1(\pi)$ and call the $1$-complement of $\pi$. This is equal to $((n+2)-\pi_1)((n+2) - \pi_2) \ldots ((n+2) - \pi_{j-1}) 1 ((n+2) - \pi_{j+1}) \ldots ((n+2) - \pi_{n})$. If $\pi = {\mathbf{623}}1{\mathbf{54}}$, then $c_{1}(\pi) = {\mathbf{265}}1{\mathbf{34}}$. A version of this definition of complementation was used in \cite{Bo21} by B\'ona to understand alternating runs in permutations.

Finally, many well-known and/or named sequences appear in this paper. In Table~\ref{tab: sequences}, we provide the names, notation, and OEIS numbers for these sequences as reference. 

\begin{table}[htbp]
    \centering
    \begin{tabular}{|c|c|c|c|}
    \hline
        Notation & Name & Enumeration description & OEIS \\ \hline \hline
        $d_n$ & Derangement numbers & derangements of $[n]$ & A000166 \\ \hline
        $B_n$ & Bell numbers & set partitions of $[n]$ & A000110 \\ \hline
        $\stirling{n}{k}$ & Stirling numbers (2nd) & $k$-block set partitions of $[n]$ & A008277 \\ \hline
        $C_n$ & Catalan numbers & Dyck paths of semi-length $n$ & A000108 \\ \hline
        $r_n$ & Riordan numbers & singleton-free noncrossing partitions of $[n]$ & A005043 \\ \hline
        $G_n$ & Gould numbers & $\S_n(1\ol{23})$ permutations starting with 21 &  A040027 \\ \hline
        $S_n$ & Large Schr\"oder numbers & separable permutations of $[n]$ & A006318 \\ \hline
    \end{tabular}
    \caption{Notation and descriptions for the well-known combinatorial sequences used in this paper.}
    \label{tab: sequences}
\end{table}

\section{Arrow patterns}\label{sec:arrows}

We define an \emph{arrow pattern} $\alpha = (\nu; H)$ of size $k$ as a string of positive integers $\nu = a_1\ldots a_m$ alongside a (possibly empty) collection of arrows $H = \{b_i \to c_i \mid 1 \leq i \leq h\}$, so that the distinct integers appearing in either $\nu$ or $H$ form the set $[k]=\{1,2,\dots,k\}$. We denote the set of all arrow patterns of size $k$ by $\A_k$.

Given a permutation $\pi\in\S_n$ and an arrow pattern $\alpha=(\nu;H)\in\A_k$ with $\nu=a_1\ldots a_m$ and $H=\{b_i\to c_i \mid 1\leq i \leq h\}$, we say that $\pi$ contains the arrow pattern $\alpha$ if there is some set $X=\{x_1<\cdots<x_k\}\subseteq[n]$ so that
\begin{itemize}
    \item there is some $t_1<\cdots <t_m$ so that $\pi_{t_1}\ldots \pi_{t_m} = x_{a_1}\ldots x_{a_m}$, and 
    \item if $\hat\pi=\theta^{-1}(\pi)$, then for each $1\leq i \leq h$, we have $\hat\pi(x_{b_i})=x_{c_i}$.
\end{itemize}

We say $\pi$ avoids $\alpha$ if it does not contain it. For example, consider the arrow pattern $\alpha=(231;1\to4)$, and the permutation $\pi = 639245781$ with $\hat\pi = (6,3)(9,2,4,5,7,8,1)$. %946573812$. 
The permutation $\pi$ contains multiple occurrences of $231$ in the classical sense, for example $392$ and $241$. Of these two options, only $241$ is part of an occurrence of the arrow pattern $\alpha$. In this case $X = \{1,2,4,9\}$ because $\hat\pi(1) = 9$ and is indeed an occurrence of $\alpha$ since 9 is the largest element of $X$, thus the ``4'' in this pattern. On the other hand, 392 cannot be part of an occurrence of $\alpha$ since $\hat\pi(2) = 4 < 9$.
As another example, $\pi = 3627154$ avoids $\alpha = (231; 1 \to 4)$. Even though there are multiple occurrences $x_2x_3x_1$ of $231$, none have the property that $x_4=\hat\pi(x_1)$ is greater than $x_i$ for each $1\leq i \leq 3.$

Let $\S_n(\alpha)$ denote the set of permutations that avoid the arrow
pattern $\alpha=(\nu;H)$ and let 
$a_n(\alpha)=|\S_n(\alpha)|.$ Note that if $H$ is empty, then $\S_n(\alpha) = \S_n(\nu)$ in the classical sense of pattern avoidance.

\renewcommand{\arraystretch}{1.2}
\begin{table}[H]
    \centering
    \begin{tabular}{|c|c|c|c|} 
    \hline
         $\alpha$  & $a_n(\alpha)$  & OEIS & Theorem \\ \hline\hline
          $(1;1\to1)$ 
        & $d_n$&  A000166 &Proposition \ref{prop: 1 1 to 1} \\ \hline
            $(i;j\to j)$ 
        & $d_n+d_{n-1}$& {A000255}  &Theorem \ref{thm:1 2 to 2} \\  \hline
         $(i;i\to j)$ & \multirow{2}{*}{1} &\multirow{2}{*}{A000012} &\multirow{2}{*}{Proposition \ref{prop: i i to j}}  \\  \cline{1-1}
          $(j;i\to j)$ & & &   \\ \hline
          $(a;b\to c)$ 
        & $2^{n-1}$& A000079 &Theorem \ref{thm:atbc} \\ \hline
        
       %    & $3\to3$    & ??& ?? &?? \\  \hline
    \end{tabular}
    \caption{The enumeration of permutations avoiding patterns $\alpha=(\nu;H)$ in $\A_1$ or $\A_2$ where $|\nu|=|H|=1$.  Here $d_n$ denotes the $n$-th derangement number. Additionally, $\{i,j\}=\{1,2\}$, and $\{a,b,c\}=\{1,2,3\}.$}
    \label{tab: A1 and A2}
\end{table}

Let us first enumerate those permutations avoiding patterns $\alpha=(\nu;H)$ in $\A_1$ or $\A_2$ where $|\nu|=1$.  A summary of these results can be found in Table~\ref{tab: A1 and A2}.
Note there are exactly two arrow patterns in $\A_1$, namely $(1; \varnothing)$ and $(1;1\to1)$. In the first case, this is equivalent to classical pattern avoidance of the pattern 1, and so $|\S_n((1;\varnothing))|=0$ for $n\geq 1$. We address the second case in the proposition below.

\begin{proposition}\label{prop: 1 1 to 1}
    For $n \geq 1$, $a_n(1;1 \to 1) = d_n$ where $d_n$ is the $n$-th derangement number.
\end{proposition}

\begin{proof}
    Permutations avoiding $(1;1\to1)$ are clearly those permutations $\pi$ for which $\hat\pi$ has no fixed points. Since the $n$-th derangement number is exactly the number of permutations with no fixed points, the result follows.
\end{proof}

Now we consider all elements $(\nu; H)\in\A_2$ with $|\nu|=1$. First note that if $H=\varnothing,$
this is just classical pattern avoidance of a pattern $\nu$ of size 2 in which case $a_n(\nu) = 1$ for all $n$. Below we consider the cases in which $|H|=1$. All cases in which $|H|> 1$ can easily be reduced to one of these cases or to one of the cases above, $(1;1\to1)$ or $(1;\varnothing)$.

\begin{theorem}\label{thm:1 2 to 2}
    For $n \geq 2$, $a_n(1; 2 \to 2) = a_n(2;1 \to 1) = d_n + d_{n-1}$. %and $a_1(1; 2 \to 2) = 1$.
\end{theorem}

\begin{proof}
    If $\pi \in \S_n(1; 2 \to 2)$, either $\hat\pi$ has no fixed points, or if it has a fixed point it can only be $\pi_1 = 1$, because any other larger fixed point will necessarily create a $(1; 2 \to 2)$ pattern since the fixed point will be larger than some element. There are $d_n$ permutations with no fixed points and $d_{n-1}$ permutations with $1$ as the only fixed point. An identical argument holds for avoiding $(2; 1\to1)$ but with $\hat\pi$ either as a derangement or with $n$ as its only fixed point.
\end{proof}

\begin{proposition}\label{prop: i i to j}
    For $n \geq 1$, $a_n(i; i\to j) =a_n(j; i\to j) = 1$ for any $i,j\in\{1,2\}$ with $i\neq j.$
\end{proposition}

\begin{proof}
    Within any cycle of length at least 2 in $\hat\pi$, you will have both a smaller element mapping to a larger one and a larger element mapping to a smaller one. Thus the only way to avoid this is to have $\pi=\hat\pi$ be the permutation of only fixed points.
\end{proof}

\begin{theorem} \label{thm:atbc}
    For $n \geq 1$, $a_n(a;b \to c) = 2^{n-1}$ with $\{a,b,c\} =  \{1,2,3\}$.
\end{theorem}

\begin{proof}
    Let $\pi \in \S_n(a;b\to c)$. There are three cases to consider, two of which are almost identical: $a$ is the largest element, $a$ is the middle element or $a$ is the smallest element. First, suppose $a$ is larger than $b$ and $c$. In this case, notice if we have two cycles in $\hat\pi$ of size larger than one, we automatically get an $(a; b\to c)$ pattern. Indeed, suppose the cycles start with $\pi_i$, $\pi_k$ and end with $\pi_j$, $\pi_{\ell}$, where $i < j < k < \ell$ so that $\pi_k > \pi_i > \pi_j$. If $b < c$, $(\pi_k; \pi_j \to \pi_i)$ is a $(3; 1 \to 2)$ pattern. If $b > c$, $(\pi_k; \pi_i \to \pi_{i+1})$ is a $(3; 2 \to 1)$ pattern.

    Furthermore, if a cycle is going to have size larger than one, it must be the cycle containing $n$ and the elements in that cycle must be decreasing if $b < c$ or increasing if $b > c$. This means to construct such a $\pi$ we just need to select the elements that appear in the cycle with $n$, the rest will be fixed points. There are $2^{n-1}$ ways to do this. An identical argument works in the case that $a$ is the smallest element.

    If $a$ is the middle element, we instead obtain a bijection with compositions of $n$. We consider the case $b < a < c$ since the other case $c < a< b$ is identical. First we claim that every cycle $(\pi_i \pi_{i+1} \cdots\pi_{i+k})$ must have $\pi_{i+1} < \pi_{i+2} < \cdots < \pi_{i+k} < \pi_i$. If there is some descent, say $\pi_{j} > \pi_{j+1}$ then we either have $\pi_{j+2} > \pi_j$ in which case $(\pi_j; \pi_{j+1} \to \pi_{j+2})$ is a $(2; 1 \to 3)$ pattern, or $\pi_{j+2} < \pi_j$. We can then iterate this argument asking if $\pi_{j+\ell} < \pi_j$ until eventually we reach $\pi_{i+k}$. If $\pi_{i+k} < \pi_j$, then $(\pi_j; \pi_{i+k} \to \pi_i)$ is a $(2; 1 \to 3)$ pattern, if $\pi_{i+k} > \pi_{j}$, then $(\pi_j; \pi_{i+k-1} \to \pi_{i+k})$ is a $(2;1 \to 3)$ pattern.

    Let $\epsilon_k = k123\cdots(k-1)$, we claim that $\pi = \epsilon_{k_1} \oplus \epsilon_{k_2} \oplus \cdots \oplus \epsilon_{k_m}$ for $k_1 + k_2 + \cdots + k_m = n$, a composition of $n$. If not, this implies there exists some $\pi_j$ in $\pi$ that is smaller than at least two left-to-right maxima of $\pi$, $\pi_{\ell_1}$ and $\pi_{\ell_2}$, with $\ell_1 < \ell_2 < j$. A similar argument to the one in the previous paragraph shows that either $(\pi_{\ell_1}; \pi_j \to \pi_{j+1})$ is a $(2; 1\to 3)$ or we eventually get a $(2; 1\to 3)$ pattern with $(\pi_{l_1}; \pi_{j+k} \to \pi_{\ell_2})$. This shows that such $\pi$ are in bijection with compositions of $n$, which are also enumerated by $2^{n-1}$.
\end{proof}

Some arrow patterns are naturally vincular patterns, as stated in the lemma below. Recall that the bars above two adjacent elements in the pattern imply that those two elements are adjacent in an occurrence of that pattern in the permutation. 

\begin{lemma} \label{lem: vincArrow}
    Given $\nu=\nu_1\ldots \nu_{k-1}$, if $b\in [k]$ with $b\neq \nu_j$ for any $j$ and with $b > \nu_i$ for some $i$, then \[\S_n(\nu ; b \to \nu_i) = \S_n(\nu_1\ldots \nu_{i-1}\overline{b \nu_i} \nu_{i+1}\ldots \nu_k).\]
\end{lemma}

\begin{proof}
    Recall that in $\hat\pi,$ each cycle begins with its largest element and that these elements are left-to-right maxima in $\pi.$ If $b \to \nu_i$, this means $b$ and $\nu_i$ are in the same cycle of $\hat\pi$. There are then two options to consider: either $b$ is the final entry in its cycle, or it is not. The first case cannot occur. Indeed, if $b$ is the final entry in its cycle, this means $\nu_i$ is the first entry of the same cycle, and so is larger than $b$. This contradicts the assumption that $b > \nu_i$.

    In the second case, since $b$ is not the final entry of its cycle, in order to have $b \to \nu_i$, we must have that $\nu_i$ is the entry to the right of $b$ in $\pi$. This proves the result.
\end{proof}

For example, avoiding the arrow pattern $\alpha=(341; 2\to 1)$ is equivalent to avoiding the vincular pattern $34\overline{21}$ since in $\hat\pi$, we start each cycle with its largest element, so in order for the ``2'' to map to the ``1'' in $\hat\pi$, we need the ``1'' to come immediately after the ``2'' in the cycle form (as opposed to wrapping around to the beginning of the cycle).

We say that two arrow patterns $\alpha$ and $\beta$ are \emph{arrow-Wilf equivalent} if $|\S_n(\alpha)| = |\S_n(\beta)|$. One general example of arrow-Wilf equivalence is the following. Recall that $c_1(\pi)$ is the $1$-complement of $\pi$, where every element except 1 is complemented.

\begin{proposition} \label{prop: rc1}
    Let $n\geq 1$ and $k\geq 1$. If $\alpha =(\nu; k\to 1)\in\A_k$ with $\nu\in\S_{k-1}$, then
    \[a_n(\nu;k\to 1) = a_n(c_1(\nu^r); k\to 1).\]
\end{proposition}

\begin{proof}
    First, if $\nu = \nu_1\nu_2\ldots \nu_{j-1} 1 \nu_{j+1} \ldots \nu_{k-1}$, by Lemma \ref{lem: vincArrow} avoiding $(\nu;k \to 1)$ is equivalent to avoiding $\nu' = \nu_1\ldots \nu_{j-1} \ol{k 1} \nu_{j+1} \ldots \nu_{k-1}$.

    So $|\S_n(\nu')| = |\S_n((\nu')^{rc})|$ by taking reverse complements. However, notice that \[(\nu')^{rc} = (k+1-\nu_{k-1})(k+1-\nu_{k-2}) \ldots (k+1-\nu_{j+1}) \ol{k1} (k+1-\nu_{j-1})\ldots (k+1-\nu_1).\] But by Lemma \ref{lem: vincArrow}, avoiding this pattern is precisely equivalent to avoiding \[((k+1-\nu_{k-1})(k+1-\nu_{k-2}) \ldots (k+1-\nu_{j+1}) 1 (k+1-\nu_{j-1})\ldots (k+1-\nu_1);k \to 1),\]
    which is equal to $ (c_1(\nu^r);k \to 1)$. In all, we conclude that $a_n(\nu; k \to 1) = |\S_n(\nu')| = |\S_n((\nu')^{rc})| = a_n(c_1(\nu^r);k \to 1)$. \end{proof}

For example, consider the arrow pattern $\alpha = (24153; 6 \to 1)$. Avoiding this pattern is equivalent, by Lemma \ref{lem: vincArrow}, to avoiding $\nu' = 24\ol{61}53$. By taking usual reverse complements, $|\S_n(24\ol{61}53)| = |\S_n(42\ol{61}35)|$ since $(\nu')^r = 35\ol{16}42$, so that $(\nu')^{rc} = 42\ol{61}35$. Avoiding this pattern is equivalent to avoiding the arrow pattern $(42135; 6 \to 1)$, but notice $42135 = c_1((24153)^r)$. So $\alpha = (24153; 6 \to 1)$ and $\beta = (42135; 6 \to 1)$ are arrow-Wilf equivalent.

\section{Permutations that avoid $(12;b\to c)$} \label{sec:12tbc}

In this section, we consider permutations avoiding an arrow pattern $\alpha =(12; b\to c)$ in which either $b$ or $c$ (or both) is contained in $\{1,2\}$. Table \ref{tab: 12} summarizes the results in this section:

\renewcommand{\arraystretch}{1.2}
\begin{table}[H]
    \centering
    \begin{tabular}{|c|c|c|c|c|} 
    \hline
         $\nu$ & $H$ & $a_n(\nu;H)$  & OEIS & Theorem \\ \hline\hline
        \multirow{9}{*}{12}  
         & $1\to1$ 
        & $d_n+d_{n-1}$& A000255  &Proposition \ref{prop:12t11} \\ \cline{2-5}
           & $2\to2$ 
        & $n!-d_n$& A002467  &Theorem \ref{thm:12t22} \\ \cline{2-5}
        & $1\to2$ 
        &  $B_n$  & A000110 &Theorem \ref{thm:12t12}\\ \cline{2-5}
           & $2\to1$ 
        & $n!$ & A000142 &Proposition \ref{prop:12t21}  \\ \cline{2-5}
           & $1\to3$ 
        & $S_{n-1}$  & A006318 &Theorem \ref{thm:12t13} \\ \cline{2-5}
           & $3\to1$ 
        &  $C_n$ &  A000108 &Proposition \ref{prop:12t31} \\ \cline{2-5}
           & $2\to3$ 
        & $2B_{n-1}$& A186021 &Theorem \ref{thm:12t23} \\ \cline{2-5}
           & $3\to2$ 
        & $B_n$ & A000110 &Proposition \ref{prop:12t32} \\  \hline 
       %    & $3\to3$    & ??& ?? &?? \\  \hline
    \end{tabular}
    \caption{The enumeration of those permutations that avoid the arrow pattern $\alpha=(12;H)$ for certain $H$. Here $B_n$ represents the $n$-th Bell number, $d_n$ represents the $n$-th derangement number, $S_n$ represents the $n$-th Large Schr\"oder number, and $C_n$ the $n$-th Catalan number.}
    \label{tab: 12}
\end{table}

\begin{proposition} \label{prop:12t11}
    For $n\geq 2$, $a_n(12; 1 \to 1) = d_n + d_{n-1}$
\end{proposition}

\begin{proof}
    This follows immediately from noticing that for $\pi \in \S_n(12;1\to1)$, if $\hat\pi$ has a fixed point that is less than $n$, it will necessarily appear before $n$ in $\pi$, and thus result in a $(12; 1\to1)$ pattern. This means that $\hat\pi$ either has no fixed points, or has only $n$ as a fixed point. The first case is enumerated by $d_n$, the second by $d_{n-1}$.
\end{proof}

Next we consider the case where the fixed point is the larger element in an increase,

\begin{theorem} \label{thm:12t22}
    For $n \geq 1$, $a_n(12; 2 \to 2) = n! - d_n$.
\end{theorem}

\begin{proof}
    It suffices to show $a_n = (n-1)(a_{n-1} + a_{n-2})$. Let $\pi \in \S_n(12;2 \to 2)$, we cannot have $\pi_n = n$. The number of such $\pi$ with $\pi_{n-1} = n$ is precisely $(n-1) a_{n-2}$. Indeed, we can obtain all such permutations by adding $nk$ with $1 \leq k \leq n-1$, to the end of any $\tau \in \S_{n-2}(12;2\to2)$, shifting all the elements appropriately. Since $\tau$ avoids $(12; 2 \to 2)$ and we added a cycle of size $2$, we could not have created a $(12; 2 \to 2)$ pattern. Conversely, removing $nk$ from the end of $\pi$ will necessarily result in a permutation $\tau \in \S_{n-2}(12 ; 2 \to 2)$.

    The number of $\pi$ with $\pi_k = n$ for $1 \leq k \leq n-2$ is $(n-1)a_{n-1}$. Indeed, consider any $\tau \in \S_{n-1}(12;2\to2)$, form a new permutation $\tau' \in \S_n$ by replacing $n-1$ in $\tau$ with $n$ and inserting $r$ with $1 \leq r \leq n-1$ to the right of $n$, shifting up other entries accordingly. For example, if $\tau = 25143$, we replace $5$ with $6$, and can place $r = 3$ to the right of $6$ to obtain $\tau' = 263154 \in \S_6(12; 2\to2)$. The resulting permutation $\tau'$ necessarily avoids $(12; 2 \to2)$ because $\tau$ did and we added an entry to the right of $n$ which does not introduce a fixed point. Furthermore, $\tau'_k = n$ for $1 \leq k \leq n-2$ because we know that $n-1$ is not in position $n-1$ of $\tau$. This process is easily reversible, so produces all $\pi \in \S_n(12; 2\to2)$ with $\pi_k = n$ for $1 \leq k \leq n-2$, which are therefore enumerated by $(n-1)a_{n-1}$. This covers all possible positions of $n$, so $a_n = (n-1)(a_{n-1} + a_{n-2})$ as desired.
\end{proof}

\begin{theorem} \label{thm:12t12} For $n\geq 1,$
    $a_n(12;1\to2) = B_n$, the $n$-th Bell number. %\cite[A000110]{OEIS}.
\end{theorem}

\begin{proof}
    Within any cycle $(c_1,c_2,\dots,c_m)$ of $\hat\pi,$ we must have $c_1>c_2>\cdots>c_m$. Indeed, if we have $c_i<c_{i+1}$ for some $i$, then $c_ic_{i+1}$ is a 12 pattern in $\pi$ and $\hat\pi(c_i)=c_{i+1}$, and so is an occurrence of the arrow pattern $(12, 1\to 2).$ Since every cycle in $\hat\pi$ must be written in decreasing order, the number of such permutations is exactly the number of set partitions of $[n],$ given by the $n$-th Bell number. 
\end{proof}

\begin{proposition} \label{prop:12t21} 
    For $n\geq 1,$
    $a_n(12; 2 \to 1) = n!$
\end{proposition}

\begin{proof}
    It is impossible for an occurrence of $12$ to also satisfy $2 \to 1$. Indeed, suppose $\pi_i < \pi_j$ for $i < j$ and $\pi_j \to \pi_i$. Note $\pi_i$ cannot start a cycle containing $\pi_j$ in $\hat\pi$ since each start of a cycle is the largest element of its cycle. This forces $i = j+1$ which is a contradiction. 
\end{proof}

\begin{theorem} \label{thm:12t13}
    For $n\geq 1$, 
    $a_n(12;1\to3)=S_{n-1}$, the $(n-1)$-st large Schr\"oder number.
\end{theorem}

\begin{proof}
    Let us denote $a_n:=a_n(12; 1\to 3)$. It is enough to show that $a_n$ satisfies the large Schr\"oder recurrence: 
      \[
    a_n = a_{n-1} + \sum_{i=1}^{n-1} a_i a_{n-i}.
    \]
    First, let us note that the number of permutations avoiding $(12;1\to 3)$ that end with $n$ is equal to $a_{n-1}$ since removing or adding $n$ as a fixed point cannot introduce a $(12;1\to3)$ pattern. 
    
     Next, let us consider those permutations that do not end with $n$. If $\tau\in \S_k(12;1\to 3)$ and $\eta \in \S_{n-k}(12;1\to 3)$, then if $\tau = \tau_1\ldots \tau_i k \tau_{i+2}\ldots \tau_k$, we can define a permutation \[\pi = \tau_1\ldots \tau_i \eta_1'\ldots \eta_{n-k}' \tau_{i+2}\ldots \tau_k k\] where $\eta_i' = \eta_i + k.$ We claim $\pi \in \S_n(12;1\to 3).$  Indeed, if $\pi$ had a $(12;1\to 3)$ pattern, it must have the ``1'' be $\tau_i$ for some $i$ and the ``3'' be $\eta_j'$ for some j. However, the only way to have $1\to3$ in this case is if $k$ is the ``1''. But then, there cannot be a $12$ pattern. 
     
     Let us see that this process is always reversible. Take $\pi\in\S_n(12;1\to 3)$ with $\pi_n=k$ and $\pi_j=n$. If there is any $m>j$ with $\pi_m<k$, then $\pi_m,\pi_{m+1}, \ldots, \pi_{n-1}$ must all be less than $k$ or else one would have a $(12;1\to3)$ pattern. Similarly, if there is an $r<j$ with $\pi_r<k$, then everything in the same cycle of $\hat\pi$ as $\pi_r$ must also be less than $k$ or one would have a $(12;1\to3)$ pattern. Therefore $\pi$ is of the form \[\pi = \tau_1\ldots \tau_i \eta_1'\ldots \eta_{n-k}' \tau_{i+2}\ldots \tau_k k\] where we necessarily have $\eta\in\S_{n-k}$ and $\tau = \tau_1\ldots \tau_{i}k\tau_{i+1}\ldots\tau_{k}\in\S_k(12;1\to3).$ Thus the recursion holds and the theorem is proven.
\end{proof}

\begin{example}\label{ex: 12t13}
    Consider $\pi = 23{\bf 6879}145 \in \S_9(12;1\to3)$. This corresponds to the permutations $\tau = 23514 \in \S_5(12;1 \to 3)$ and $\eta = 1324 \in \S_4(12;1\to3)$. Here, the entries of $\eta$ are all shifted up by $5$, inserted as the bold positions where $5$ is in $\tau$, and $5$ is moved to the end of the permutation.
    
    Notice, also, that $\eta$ is an example of a permutation with its largest entry last, which corresponds to the permutation $132 \in \S_3(12;1\to3)$.
\end{example}

\begin{proposition} \label{prop:12t31}
    For $n\geq 1$, $a_n(12;3\to1) = C_n$, the $n$-th Catalan number.
\end{proposition}

\begin{proof}
    First, notice $\S_n(12; 3 \to 1) = \S_n(\overline{31}2)$ by Lemma \ref{lem: vincArrow}. But $|\S_n(\overline{31}2)| = C_n$ by Lemma \ref{lem: vincAvoidance}, which concludes the proof.
\end{proof}

\begin{theorem} \label{thm:12t23}
    For $n \geq 2$, $a_n(12; 2 \to 3) = 2B_{n-1}$ and $a_1(12; 2 \to 3) = 1$.
\end{theorem}

\begin{proof}
    Suppose $n \geq 3$ and $\pi \in \S_n(12; 2 \to 3)$. Let $b_n$ denote the number of these permutations with $\pi_n = 1$. First, notice that $a_n = a_{n-1} + b_n$. Indeed, $\pi \in \S_n(12; 2 \to 3)$ either has $\pi_n = n$ or $\pi_n = 1$. If not, then $\pi_n = k$ for some $2 \leq k \leq n-1$. Suppose $\pi_j = 1$ and $\pi_\ell = n$, then $(\pi_j\pi_n;\pi_n \to \pi_\ell)$ is a $(12; 2 \to 3)$ pattern. There are $a_{n-1}$ such permutations with $\pi_n = n$ and by definition $b_n$ with $\pi_n = 1$.

    We will now find a recursion for $b_n$. So for now, we only consider $\pi \in \S_n(12;2\to3)$ with $\pi_n = 1$. We break these permutations up into two groups. The first ends with $n > \pi_{i_1} > \pi_{i_2}> \ldots > \pi_{i_k} > 1$. The second group has some ascent between $n$ and $1$. 

    To enumerate the first group, we can select any $k$ elements from $2,\dots,n-1$ to be the decreasing portion between $n$ and $1$. The elements before $n$ can then be arranged in $a_{n-k-2}$ ways. So the first group of permutations is enumerated by $\sum_{k=0}^{n-2} \binom{n-2}{k} a_{n-k-2}$.

    To enumerate the second group, consider the final ascent that occurs between $n$ and $1$. Say this ascent is $i < j$ so that $\pi = \pi_1 \pi_2 \cdots \pi_m ij \pi_{m+3} \cdots \pi_n$ with $\pi_{m+3} > \cdots > \pi_n = 1.$ To avoid $(12;2\to3)$ we must have all the elements less than $i$ occurring to the right of $j$. If not, there is some element $k < i$ to the left of $j$, and therefore $i$, so $(ki; i \to j)$ is a $(12; 2\to 3)$ pattern. Notice this only works because we necessarily have $i \to j$ since the ascent appears after $n$. But since $i < j$ is the final ascent in $\pi$, all elements after $j$ are decreasing, so $\pi$ must end with $(i-1)(i-2)\ldots21$.
    
    If we fix $2 \leq i \leq n-2$ that begins the ascent, we can enumerate these permutations by choosing the elements between $i$ and $i-1$ in $\pi$, which must be decreasing, non-empty and in $\{i+1,\dots,n-1\}$. If there are $k$ such elements, we can choose them in $\binom{n-i-1}{k}$. The elements to the left of $i$ are enumerated by $b_{n-i-k+1}$ since ${\rm red}(\pi_1 \cdots \pi_m i)$ must be in $\S_n(12; 2\to3)$ and  $i$ reduces to $1$.
    
    For example, the permutation $784{\bf 65} 321 \in \S_7(12; 2\to3)$ has final ascent beginning with $i = 4$, we then choose $6$ and $5$ to appear after $4$, and are forced to end the permutation with $321$. After removing all the elements to the right of $i$ and reducing we find ${\rm red}(784) = 231 \in \S_3(12; 2\to3)$ with final entry $1$.

    Summing over all choices of $i$, there are
    \[
    \sum_{i=2}^{n-2} \sum_{k=1}^{n-i-1} \binom{n-i-1}{k} b_{n-i-k+1}
    \]
    such permutations. This means
    \[
    b_n = \sum_{k=0}^{n-2} \binom{n-2}{k} a_{n-k-2} + \sum_{i=2}^{n-2} \sum_{k=1}^{n-i-1} \binom{n-i-1}{k} b_{n-i-k+1}.
    \]
    We can simplify the first and second sums to find:
    \[
    b_n = -a_{n-2} + a_1 + \sum_{j=0}^{n-2} \binom{n-2}{j} a_{j} + \sum_{m=2}^{n-2} \left[\binom{n-2}{m-1}\right]b_m.
    \]
    Since $b_n = a_n - a_{n-1}$, we can express this as a recursion in $a_n$,
    \begin{align*}
    a_n &= a_{n-1} - a_{n-2} + a_1 + \sum_{j=0}^{n-2} \binom{n-2}{j} a_j + \sum_{m=2}^{n-2} \binom{n-2}{m-1}(a_m - a_{m-1})\\
    &= 1 + \sum_{m=1}^{n-1} \binom{n-2}{m-1} a_m
    \end{align*}
    This recursion implies the desired result.
\end{proof}

\begin{remark}
    Given how simple the recursion is, it would be interesting to find a more straightforward way to prove this result.
\end{remark}

\begin{proposition} \label{prop:12t32}
    For $n\geq 1$, $a_n(12; 3 \to 2) = B_n$.
\end{proposition}

\begin{proof}
    First notice by Lemma \ref{lem: vincArrow}, avoiding the arrow pattern $(12; 3 \to 2)$ is equivalent to avoiding the vincular pattern $1 \overline{32}$. But this is $B_n$ by Lemma \ref{lem: vincAvoidance}.
\end{proof}

\section{Permutations that avoid $(21;b\to c)$}\label{sec:21tbc}

In this section, we will consider those permutations that avoid the arrow pattern $(21; b\to c)$ for some $b,c\in\{1,2,3\}.$

\renewcommand{\arraystretch}{1.2}
\begin{table}[H]
    \centering
    \begin{tabular}{|c|c|c|c|c|} 
    \hline
         $\nu$ & $H$ & $a_n(\nu;H)$  & OEIS & Result \\ \hline\hline
        \multirow{9}{*}{21}  & $1\to1$ 
        & $n!$& A000142 & Proposition \ref{prop: 21t11} \\ \cline{2-5}
           & $2\to2$ 
        & $d_n+\sum_{i=1}^{n} d_{n-i}a_{i-1}(21;2\to 2)$& A259870 & Theorem \ref{thm: 21t22} \\ \cline{2-5}  & $1\to2$ 
        &  \multirow{2}{*}{1}  & \multirow{2}{*}{A000012}& \multirow{2}{*}{Proposition \ref{prop: 21t12}} \\
           & $2\to1$ 
        & & & \\ \cline{2-5}
           & $1\to3$ 
        & $2^{n-1}$ & A000079 & Theorem \ref{thm: 21t13} \\ \cline{2-5}
           & $3\to1$ 
        &  $C_n$ &  A000108 & {Proposition} \ref{prop: 21t31} \\ \cline{2-5}
           & $2\to3$ 
        &$|\S_n(\ol{41}\; \ol{32}, \ol{32} \; \ol{41})|$ & A074664 & Theorem \ref{thm: 21t23} \\ \cline{2-5}
           & $3\to2$ 
        & $B_n$ & A000110 & Proposition \ref{prop: 21t32} \\ 
            \hline
        %   & $3\to3$     & & ?? \\  \hline
        
    \end{tabular}
    \caption{The enumeration of those permutations that avoid the arrow pattern $\alpha=(21;H)$ for certain $H$. Here $B_n$ represents the $n$-th Bell number, $d_n$ represents the $n$-th derangement number, and $C_n$ represents the $n$-th Catalan number.}
    \label{tab:21}
\end{table}

\begin{proposition} \label{prop: 21t11}
    For $n \geq 1$, we have
    \[
    a_n(21;1 \to 1) = n!
    \]
\end{proposition}

\begin{proof}
    Since any fixed point in $\hat\pi$ must be a left-to-right maximum in $\pi$, it is not possible for the pattern $(21; 1 \to1)$ to occur, so every permutation avoids it.
\end{proof}

\begin{theorem} \label{thm: 21t22}
    For $n\geq 1,$ we have \[a_n(21;2\to 2) = d_n+\sum_{i=1}^{n} d_{n-i}a_{i-1}(21;2\to 2).\]
\end{theorem}

\begin{proof}
    First, notice that if $\hat\pi$ is a derangement, then $\pi$ avoids $(21;2 \to 2)$ because we never have $2 \to 2$. Now, consider the case where $\hat\pi$ has at least one fixed point. For $\pi$ to avoid $(21; 2 \to 2)$ we must have every element to the left of the fixed point smaller than every element to the right of the fixed point. Indeed, if $i$ is a fixed point of $\hat\pi,$ then $i$ is a left-to-right maximum in $\pi$, so it is certainly the case that all elements larger than $i$ appear to its right. However, if any element smaller than $i$ appears to its right, it forms a $(21;2\to2)$ pattern with $i$.

    Suppose $i$ is the maximal fixed point in $\hat\pi$. Then the elements to its right can be any derangement of size $n-i$. The elements to the left of $i$ are enumerated by $a_{i-1}$ since they, too, must avoid $(21;2 \to 2)$. This establishes the desired recursion.
\end{proof}

\begin{remark}
    We note that the original recurrence for this OEIS sequence is
    \[
    a_n = n a_{n-1} + (n-2) a_{n-2} - \sum_{j=1}^{n-1} a_j a_{n-j}.
    \]
    Theorem \ref{thm: 21t22} implies that the OEIS sequence can be expressed in terms of derangement numbers as:
    \[
    a_n = d_{n-1} + \sum_{i=1}^{n-1} d_{n-1-i} a_i.
    \]
\end{remark}

\begin{proposition} \label{prop: 21t12}
    For $n\geq 1,$ we have $a_n(21;1\to 2)=a_n(21; 2\to1) = 1.$
\end{proposition}
\begin{proof}
    In both cases, only the increasing permutation satisfies the requirement. 
\end{proof}

\begin{theorem} \label{thm: 21t13}
    For $n\geq 1$, $a_n(21; 1\to 3) = 2^{n-1}.$
\end{theorem}

\begin{proof}
    We will argue these permutations are in bijection with compositions of $n$. Let $\pi \in \S_n(21; 1 \to 3)$. First, if $\hat\pi = c_1c_2\cdots c_k$ where each $c_i$ is a cycle, we claim the elements of $c_i$ are all less than the elements of $c_{i+1}$ for $i = 1,\dots, k-1$. If not, consider the two cycles $(a_1,\dots,a_m)$ and $(b_1,\dots,b_\ell)$. If we have $a_j > b_i$, then if $i < \ell$ and $b_{i+1} > a_j$, $a_j b_i b_{i+1}$ forms a $(21; 1 \to 3)$ pattern. So this means if $i < \ell$, $a_j > b_{i+1}$. Iterate this argument until $i = \ell$ so that $a_j > b_\ell$. However, then $a_j b_\ell; b_\ell \to b_1$ forms a $(21; 1 \to 3)$ pattern in $\pi$.
    
    Now consider any cycle in $\hat\pi$, $(a_1,a_2,\dots,a_k)$. To avoid $(21; 1 \to 3)$, we must have $a_2 < a_3 < \cdots < a_k < a_1$. By definition, $a_1$ is maximal. If we do not have $a_2 < a_3 < \cdots < a_k$, there is some decrease in the cycle. By an identical argument to the previous paragraph, this decrease results in a $(21; 1 \to 3)$ pattern.
    
    The previous two paragraphs immediately imply that to construct $\pi \in \S_n(21; 1 \to 3)$, we only need to specify the size of each cycle in $\hat\pi$, given by a composition of $n$. For example, $(3,2,4)$ corresponds to $\hat\pi = (312)(54)(9678)$. This establishes the result.
\end{proof}

\begin{proposition}\label{prop: 21t31}
    For $n \geq 1$, $a_n(21;3 \to 1) = C_n$.
\end{proposition}

\begin{proof}
    This follows immediately from Proposition \ref{prop:12t31} and Proposition \ref{prop: rc1} because $C_n = a_n(12;3\to1) = a_n(21;3\to1)$.
\end{proof}

\begin{theorem}\label{thm: 21t23}
    For $n \geq 1$, $a_n(21; 2 \to 3) =  \sum_{k=0}^n \binom{n}{k} f_{n-k}$ where  $f_n = B_n - \sum_{k=1}^n f_{k-1}B_{n-k}.$%|\S_n(\ol{41}\; \ol{32}, \ol{32} \; \ol{41})|$
\end{theorem}

\begin{proof}
    First, we note that $\pi \in \S_n(21; 2\to3)$ with $\hat\pi$ having no fixed points if and only if $\pi \in \S_n(\ol{23}1)$ with $\hat\pi$ having no fixed points. Indeed, say $\pi$ contains a $(21; 2\to3)$ pattern realized by $(ba; b \to c)$. If $c$ is not a left-to-right maximum in $\pi$, then this occurs if and only if $\pi$ contains $\ol{bc} a$, which is a $\ol{23}1$ pattern. If $c$ is a left-to-right maximum in $\pi$, this means that $b$ appears at the end of a cycle in $\hat\pi$, but it cannot be the last entry in $\pi$ because $a$ occurs after it. Therefore, the entry to the right of $b$, say $d$, is another left-to-right maximum of $\pi$. In this case $\ol{bd}a$ is a $\ol{23}1$ pattern in $\pi$. Conversely, if such a $\pi$ contains a $\ol{23}1$ realized by $\ol{bd}a$ where $d$ is a left-to-right maximum in $\pi$, since $b$ occurs in a cycle of size at least $2$ there is some $c > b$ with $b \to c$, so we have that $(ba; b\to c)$ is a $(21; 2\to3)$ pattern in $\pi$.
    
    So to find all $\pi \in \S_n(21; 2 \to 3)$, it suffices to add fixed points in all possible ways to permutations $\pi \in \S_k(\ol{23}1)$ where $\hat\pi$ has no fixed points. Notice, however, that adding any fixed point to $\hat\pi$ in any $\pi$ that is $\ol{23}1$ avoiding does not create a $(21; 2\to 3)$ pattern. Indeed, the fixed point is a left-to-right maximum in $\pi$, so cannot be the $1$ in a $(21; 2\to 3)$, it cannot be the $2$ or $3$ because it is fixed in $\hat\pi$. So if $f_n$ is the number of $\pi \in \S_n(\ol{23}1)$ where $\hat\pi$ has no fixed points, this implies that
    \[
    a_n(21; 2 \to 3) = \sum_{k=0}^n \binom{n}{k} f_{n-k}.
    \]
    This establishes the first part of our result. We now have to find $f_n$. To do this, we first note that $|\S_n(\ol{23}1)| = B_n$, by Lemma \ref{lem: vincAvoidance}. We now count the number of $\ol{23}1$-avoiding permutations with at least one fixed point and remove them. Let $\pi$ be such a permutation, and suppose $a$ is the first fixed point of $\hat\pi$, with $\pi_k = a$, so that $\hat\pi = \sigma (a)\tau$ where $\sigma$ contains no fixed points. 
    
    To avoid $\ol{23}1$, we must have that all the entries of $\tau$ are larger than $a$. If not, there is some $\tau_i$ with $\tau_i < a$, but then $\ol{a\tau_1} \tau_i$ is a $\ol{23}1$ pattern. Furthermore, since $a$ is a left-to-right max, there cannot be any entries to the left of $a$ that are larger than $a$. This means that $\{\tau_1,\dots,\tau_{n-a}\} = \{a+1,a+2,\dots,n\}$ and $\{\sigma_1,\dots,\sigma_{a-1}\} = \{1,\dots,a-1\}$. Since $\sigma$ does not contain any fixed points, and ${\rm red}(\tau)$ can be any $\ol{23}1$ avoiding permutation, we conclude the number of $\ol{23}1$ avoiding permutations with at least one fixed point is $\sum_{k=1}^n f_{k-1}B_{n-k}$, so
    \[
    f_n = B_n - \sum_{k=1}^n f_{k-1}B_{n-k}.
    \]
\end{proof}

\begin{example}
    Let us see that fixed points can indeed be added to $\pi\in\S_k(\ol{23}1)$ where $\hat\pi$ has no fixed points. If we take $\pi = 412635\in\S_6(\ol{23}1)$, we see this permutation has the property that $\hat\pi=(412)(635)$ has no fixed points. Let us choose some set of fixed points to add and notice that the resulting permutation still avoids the arrow pattern $(21;2\to3).$ Select $2, 6,$ and 8 to be fixed points. Then we shift the elements of $\hat\pi$ accordingly to $(513)(947)$, and add the fixed points to obtain $(2)(513)(6)(8)(947)$, giving us a new permutation $\pi' = 251368947\in\S_n(21;2\to3).$
\end{example}

\begin{remark}
    We note that $\S_n(\ol{23}1) \subset \S_n(21; 2 \to 3)$ is true generally, with equality occurring if we remove the permutations in $\S_n(21;2\to3)$ with fixed points. We didn't need this level of generality for the previous proof.

    Also, it appears that $|\S_n(21;2\to3)| = |\S_n(\ol{41}\; \ol{32}, \ol{32} \; \ol{41})|$. These are permutations without any nested descents. It would be nice to find a bijection. Also, it seems $f_n$ satisfies $f_n + f_{n-1} = s_{n+1}$ where $s_n$ is the number of indecomposable set partitions of $n+1$ with no singletons. It would also be interesting to find a bijection here.
\end{remark}

\begin{proposition} \label{prop: 21t32}
    For $n \geq 1$, $a_n(21;3 \to 2) = B_n$.
\end{proposition}

\begin{proof}
    By Lemma \ref{lem: vincArrow}, $a_n(21; 3 \to 2) = |\S_n(\ol{32} 1)|$. But this is equal to $B_n$ by Lemma \ref{lem: vincAvoidance}.
\end{proof}

\section{Permutations that avoid $(13;b\to c)$}\label{sec:13tbc}

In this section, we enumerate those permutations that avoid $\alpha=(13; b\to c)\in\A_3.$ A summary of the results in this section can be found in Table~\ref{tab:13}.

\begin{table}[H]
    \centering
    \begin{tabular}{|c|c|c|c|c|} 
    \hline
         $\nu$ & $H$ & $a_n(\nu;H)$  & OEIS & Theorem \\ \hline\hline
        \multirow{5}{*}{13}  & $1\to2$ 
        &  \multirow{3}{*}{$B_n$} & \multirow{3}{*}{A000110}& Theorem \ref{thm:13t12} \\ \cline{2-2} \cline{5-5}
           & $2\to1$ 
        & && Theorem \ref{thm:13t21} \\ \cline{2-2} \cline{5-5}
        & $3\to2$ 
        && & Proposition \ref{prop:13t32} \\ \cline{2-5}
           & $2\to3$ 
        & complicated recurrence & N/A & Theorem \ref{thm:13t23} \\ \cline{2-5}   
           & $2\to2$
        &  $d_n + d_{n-1} + \sum_{k=1}^{n-1} (d_{n-k} + d_{n-k-1})(k-1)!$& N/A & Theorem \ref{thm:13t22} \\ \cline{2-4} \hline
    \end{tabular}
    \caption{The enumeration of those permutations that avoid the arrow pattern $\alpha=(13;H)$ for certain $H$. Here $B_n$ represents the $n$-th Bell number, $d_n$ represents the $n$-th derangement number, and $\stirling{n}{k}$ represents the Stirling numbers of the second kind.}
    \label{tab:13}
\end{table}

The first three theorems are very straightforward.

\begin{theorem} \label{thm:13t12}
    For $n \geq 1$, $a_n(13; 1 \to 2) = B_n$.
\end{theorem}

\begin{proof}
    Given $\pi \in \S_n(13; 1 \to 2)$, consider the position of $n$. If $\pi_k = n$, then $\pi_1<\pi_2 < \cdots < \pi_{k-1}$. If not, we would have a cycle of size at least two in $\hat\pi$ that does not contain $n$ (and thus appears before $n$ in $\pi$). Let $\pi_i > \pi_j$ be the first and last entries of this cycle, respectively. In this case, $(\pi_j n ; \pi_j \to \pi_i)$ forms a $(13; 1 \to 2)$ pattern in $\pi$. 

    This implies that $\pi = \sigma n \tau$, where $\sigma_1 < \sigma_2 < \cdots < \sigma_{k-1}$ and ${\rm red}(\tau) \in \S_{n-k}(\ol{12}3)$. There are $\binom{n-1}{k-1}$ choices for $\sigma$ and since $|\S_{n-k}(\ol{12}3)| = B_{n-k}$ by Lemma \ref{lem: vincAvoidance}, there are $B_{n-k}$ choices for $\tau$. Summing over all positions of $n$ shows
    \[
    a_n(13; 1 \to 2) = \sum_{k=1}^n \binom{n-1}{k-1} B_{n-k},
    \] which is equal to $B_n.$
\end{proof}

\begin{theorem} \label{thm:13t21}
    For $n \geq 1$, $a_n(13; 2 \to 1) = B_n$.
\end{theorem}

\begin{proof}
    This proof is almost identical to the proof of Theorem \ref{thm:13t12}. By similar reasoning we can conclude that for $\pi \in \S_n(13; 2 \to 1)$ with $\pi_k = n$, $\pi = \sigma n \tau$ where $\sigma_1 < \sigma_2 < \cdots < \sigma_{k-1}$, but this time ${\rm red}(\tau) \in \S_{n-k}(\ol{21}3)$. By Lemma \ref{lem: vincAvoidance}, $|\S_n(\ol{21}3)| = B_n$, so the result follows.
\end{proof}

\begin{proposition} \label{prop:13t32}
    For $n \geq 1$, $a_n(13; 3 \to 2) = B_n$.
\end{proposition}

\begin{proof}
    By Lemmas \ref{lem: vincArrow} and \ref{lem: vincAvoidance}, $a_n(13; 3 \to 2) = |\S_n(1\ol{32})| = B_n$.
\end{proof}

The next theorem is more complex and relies on the result from \cite[Porism 1]{Cl01} that the number of $1\overline{23}$-avoiding permutations of size $n$ with $m$ left-to-right minima is equal to the Stirling number of the second kind, denoted by $\stirling{n}{m}.$

\begin{theorem} \label{thm:13t23}
    For $n \geq 1$ and $a_n:=a_n(13;2 \to 3)$, \[a_n= a_{n-1}+b_n + \sum_{r=2}^{n-1}\sum_{k=0}^{n-r-1}\binom{n-r-1}{k}b_{n-k-r+1}c_{r,k} \]
    where \[b_n=\sum_{i=2}^{n-1}\binom{n-3}{n-i-1}(b_{i}+a_{i-1}) \quad \text{ and } \quad
c_{r,k}=\sum_{m=2}^{r-1} \stirling{r-2}{m-1}m^k.
\]
\end{theorem}

\begin{proof}
For the duration of this proof, let us write $a_n:=a_n(13;2\to3)$ and use $b_n$ to denote those permutations avoiding $(13;2\to3)$ that end with $1.$ 

Let us first consider the case where $\pi\in\S_n(13;2\to3)$ ends with $\pi_n=1$. First suppose $2$ precedes $n$ in $\pi.$ Then since $\pi$ avoids $(13;2\to3)$, any elements between $n$ and $1$ must be in decreasing order since otherwise for some $r<s$, $(2s;r\to s)$  would be an occurrence of that pattern. Furthermore, the elements that occur after $n$ could be any elements between 2 and $n$ and if $n$ is in position $2\leq i \leq n-1$, then the permutation $\text{red}(\pi_1\ldots \pi_{i-1})$ is any permutation that avoids the arrow pattern $(13;2\to3)$. Thus there are $\binom{n-3}{n-i-1}a_{i-1}$ such permutations where $\pi_i=n$ for any $2\leq i\leq n-1$. If instead $2$ comes after $n$ in $\pi$, then any elements between 2 and 1 must be in decreasing order for a similar reason. Additionally, removing all elements after 2 leaves us with a permutation avoiding $(13;2\to3)$ ending in 1. If $2$ is in position $i$ for $2\leq i \leq n-1$, then there are exactly $\binom{n-3}{n-i-1}b_{i}$ such permutations. Thus
\[b_n=\sum_{i=2}^{n-1}\binom{n-3}{n-i-1}(b_{i}+a_{i-1}).\]

Now let us consider those that do not end with 1. Either $\pi_n=n$, in which case there are clearly $a_{n-1}$ such permutations, or $\pi_n=r$ for some $2\leq r\leq n-1.$
 Note that for all $j<r,$ $j$ must appear after $n$. If not, then $(jn;r\to n)$ would be a $(13;2\to3)$ pattern. Furthermore, the elements $\{1,\ldots, r-1\}$ must appear in an order that avoids $1\overline{23}.$ If there were an occurrence of $1\overline{23}$, say $\pi_{i_1}\pi_{i_2}\pi_{i_3}$ with elements only from $[r-1]$, then $(\pi_{i_1}\pi_{i_2+1};\pi_{i_2}\to\pi_{i_2+1})$ would be an occurrence of $(13;2\to3)$ since either $i_2+1=3$ or $\pi_{i_2+1}>r.$ Let's say that the permutation formed by the elements of $[r-1]$ is $\sigma$. Any additional elements that occur after the first occurrence of some element $j\in[r-1]$, must occur only after a left-to-right minimum of the corresponding permutation $\sigma$ and the elements in each of these segments must be decreasing. Finally, $\sigma$ would have to end in 1 since otherwise, $(1r;\pi_{n-1}\to r)$ would be a $(13;2\to3)$ pattern.
 
 For example, if $r=6$ and $\sigma=45231,$ then $\pi$ must be of the form $
 \pi = \ldots n \ldots 4\ldots 52\ldots  31\ldots 6.$
 Indeed, it would not be possible to have any elements $m$ between 5 and 2 since $(4m;5\to m)$ would be a $(13;2\to3)$ pattern. Furthermore, any elements between 4 and 5 or between 2 and 3 would have to be in decreasing order in order to avoid $(13;2\to3).$

 Taking this into account, say that there are $k$ elements greater than $r$ that occur after the first occurrence of an element $j\in[r-1].$ Then the $1\overline{23}$-avoiding permutation $\sigma$ formed by the elements of $[r-1]$ ends in 1. There are $\stirling{r-2}{m-1}$ such permutations that have $m$ left-to-right minima (including the 1 at the end), $\binom{n-r-1}{k}$ choices for the $k$ elements between $r$ and $n$, and $m^k$ ways to distribute those elements after any left-to-right minima. Finally, deleting all but the first occurrence of an element $j\in[r-1],$ we are left with a $(13;2\to3)$-avoiding permutation ending in 1 of size $n-k-r+1.$

Taken altogether, we have 
 \[a_n = a_{n-1}+b_n + \sum_{r=2}^{n-1}\sum_{k=0}^{n-r-1}\binom{n-r-1}{k}b_{n-k-r+1} \sum_{m=2}^{r-1} \stirling{r-2}{m-1}m^k.\]
\end{proof}

\begin{remark}
    The $c_{r,k}$ are Stirling transforms of the sequence $a_m = (m+1)^k$ for a fixed $k$. For example, for $k=0,1,2,3$, the associated values are respectively $c_{r,1}=B_{r-2},$ $c_{r,2}=B_{r-1}$, $c_{r,3}=B_r - B_{r-2}$, and $c_{r,4}=B_{r+1} - 3B_{r-1} - B_{r-2}$.
\end{remark}

\begin{theorem} \label{thm:13t22}
    For $n \geq 3$, \[a_n(13;2\to2) = d_n + d_{n-1} + \sum_{k=1}^{n-1} (d_{n-k} + d_{n-k-1})(k-1)!.\]
\end{theorem}

\begin{proof}
    Given $\pi \in \S_n(13; 2\to2)$, we first note that fixed points in $\hat\pi$ must occur either as the first or last element of $\pi$. Indeed, if there is ever a fixed point in $\hat\pi$ with elements to its left and right in $\pi$, this means $\pi_k$ and $\pi_{k+1}$ are left-to-right maxima in $\pi$ with $k > 1$, then $\pi_1 \pi_k\pi_{k+1}$ is a $(13; 2\to2)$ pattern in $\pi$. There are four cases to consider. 
    
    First, if there are no fixed points in $\hat\pi$, then there clearly cannot be a $2\to2$ in $\pi.$ These are enumerated by the derangement numbers, $d_n$. Next, if the only fixed point occurs at the end of $\pi$, then $\pi_n = n$, in which case there are $d_{n-1}$ such permutations because there are no other fixed points.

    If the only fixed point in $\hat\pi$ occurs at the start of $\pi$, everything smaller than the fixed point must occur as the final entries of $\pi$ in any order to avoid a $(13; 2\to2)$ pattern. If $k$ is the fixed point, we can arrange all the entries larger than $k$ in $d_{n-k}$ ways, and arrange the entries smaller than $k$ at the end of $\pi$ in $(k-1)!$ ways. Summing over all $k$ gives $\sum_{k=1}^{n-1} d_{n-k} (k-1)!$ options.

    Finally, if there is a fixed point at the start and at the end of $\pi$, there are $\sum_{k=1}^{n-1} d_{n-k-1}(k-1)!$ options by an identical argument. The only difference being that the entries smaller than the first fixed point are not at the end of the permutation, but instead are directly to the left of $\pi_n = n$.

    Since this accounts for all possibilities, summing these cases yields the desired result.
\end{proof}

\section{Avoiding pairs of arrow patterns} \label{sec:pairs}

In this section, we enumerate those permutations that avoid the arrow patterns $\alpha = (12; b \to c) \in \mathcal{A}_3$ and $\beta = (1; 1\to1) \in \mathcal{A}_1$, meaning $\hat\pi$ also has no fixed points. We denote by $a_n(\alpha,\beta)$ the number of permutations that avoid both the arrow patterns $\alpha$ and $\beta$. A summary of the results can be found in Table \ref{tab: 12-11}.

\begin{table}[H]
    \centering
    \begin{tabular}{|c|c|c|c|c|} 
    \hline
         $\nu$ & $H$ & $a_n((\nu;H),(1;1 \to 1))$  & OEIS & Theorem \\ \hline\hline
        \multirow{5}{*}{12}  & $1\to2$ 
        &  $B_{n \geq 2}$ & A000296 & Theorem \ref{thm: 12t11t12} \\ \cline{2-5}
           & $1 \to 3$ 
        & $\sum_{i=1}^{n-1} (a_i + a_{i-1})(a_{n-i} + a_{n-i-1})$ & A052705 & Theorem \ref{thm:12t13t11} \\ \cline{2-5}
        & $3 \to 1$ 
        & $r_n$ & A005043& Theorem \ref{thm: 12t11t13} \\ \cline{2-5}
        & $2 \to 3$ 
        & $G_{n-2}$ & A040027& Theorem \ref{thm:13t23t11} \\ \cline{2-5}
        & $3\to2$ 
        & $\sum_{k=0}^{n-1} \binom{n-1}{k} a_k((12;3\to2),(1;1\to1))$ & A032347 & Theorem \ref{thm:13t32t11} \\ \cline{2-5}
        \hline
        
    \end{tabular}
    \caption{The enumeration of permutations in $\S_n(\alpha,\beta)$, for arrow patterns $\alpha = (12; b \to c)$ and $\beta = (1; 1\to1)$. Here $G_n$ represents the $n$-th Gould number, $r_n$ the $n$-th Riordan number, and $B_{n \geq 2}$ the $n$-th $2$--associated Bell number.}
    \label{tab: 12-11}
\end{table}

\begin{theorem} \label{thm: 12t11t12}
    For $n \geq 1$, $a_n((12;1 \to 2), (1; 1\to1)) = B_{n \geq 2}$, the $n$-th $2$-associated Bell number.
\end{theorem}

\begin{proof}
    We proved in Theorem \ref{thm:12t12} that $a_n(12; 1\to2) = B_n$ by showing that all the cycles in $\hat\pi$ must have their entries appear in decreasing order. Therefore the cycle decomposition of $\hat\pi$ corresponds to a partition of $n$. The additional requirement of avoiding $(1;1\to1)$ means we cannot have any fixed points. It follows that each $\pi\in \S_n((12; 1\to 2), (1;1\to1))$ corresponds bijectively to a partition of $n$ without any singletons. This is precisely counted by the 2-associated Bell numbers.
\end{proof}

\begin{theorem} \label{thm:12t13t11}
    Let $a_n := a_n((12;1 \to 3), (1; 1\to1))$. For $n \geq 1$, $a_n = \sum_{k=1}^{n-1} (a_k + a_{k-1})(a_{n-k} + a_{n-k-1})$.
\end{theorem}

\begin{proof}
    Recall that in the proof of Theorem \ref{thm:12t13}, we showed that $\pi \in \S_n(12;3 \to 1)$ with $\pi_n \not= n$ are constructed by choosing $\tau \in \S_k(12;3 \to 1)$ and $\eta \in \S_{n-k}(12;3 \to 1)$ and letting $\pi_{\tau,\eta} = \tau_1 \cdots \tau_{i} \eta_1' \cdots \eta_j' n \eta_{j+2}' \cdots \eta_{n-k}' \tau_{i+2}\cdots \tau_k k$ where $\tau_{i+1} = k$, $\eta_{j+1} = n-k$, $\eta_i' = \eta_i + k$.

    Since $\pi \in \S_n((12; 3 \to 1),(1; 1 \to 1))$ cannot have $\pi_n = n$, it suffices to consider which $\tau$ and $\eta$ correspond to $\pi_{\tau,\eta}$ with $\hat\pi_{\tau,\eta}$ avoiding fixed points.

    The only place where a fixed point could occur in $\pi_{\tau,\eta}$ is before $n$, so in $\tau_1 \cdots \tau_i \eta_1' \cdots \eta_j'$. After this, $\hat\pi$ necessarily avoids fixed points since $k$ always occurs after $n$. Notice, this implies that $\hat\pi_{\tau,\eta}$ avoids fixed points precisely when $\tau$ and $\eta$ either also avoid $(1; 1\to1)$ or only have their maximal entry fixed. Otherwise, if either contains a non-maximal fixed point, it must occur before $k$ in $\tau$ or before $n$ in $\eta$ and therefore will remain a fixed point in $\tau_1 \cdots \tau_i \eta_1' \cdots \eta_j'$.

    This means our allowable $\tau \in \S_k((12;1\to3),(1;1\to1))$ are enumerated by $a_k + a_{k-1}$, and similarly our allowable $\eta$ are enumerated by $a_{n-k} + a_{n-k-1}$. In both sums, the second term corresponds to $\tau$ or $\eta$ having its maximal point fixed.

    Summing over all possible sizes of $\tau$, we find
    \[
    a_n = \sum_{k=1}^{n-1} (a_k + a_{k-1})(a_{n-k} + a_{n-k-1})
    \]
    where $a_0 = 1$. 
    \end{proof}

\begin{example}
    In Example \ref{ex: 12t13} we considered $\pi = 236879145$, and saw that this corresponds to $\tau = 23514$ and $\eta = 1324$. In this case both $\tau$ and $\eta$ have non-maximal fixed points and all of them remain fixed points in $\pi_{\tau,\eta}$. For example, $\eta_1 = 1$, this means $\hat\pi_6 = 6$.

    If we instead consider $\tau = 32415$ and $\eta = 2143$, $\pi_{\tau,\eta} = 3241{\bf 7698}5 \in \S_9((12;1\to3),(1;1\to1))$, with $\hat\pi = (32)(41)(76)(985)$ having no fixed points. Notice that in this case $\hat\tau_5 = 5$ is fixed, but after inserting $\eta$, the $5$ becomes part of the cycle with $9$ and so is no longer fixed.
\end{example}

\begin{theorem} \label{thm: 12t11t13}
    For $n \geq 1$, $a_n((12;3 \to 1), (1; 1\to1)) = r_n$, the $n$-th Riordan number.
\end{theorem}

\begin{proof}
    For the remainder of this proof we let $a_n := a_n((12;3 \to 1),(1; 1\to1))$. By Lemma \ref{lem: vincArrow}, avoiding $(12; 3 \to 1)$ is equivalent to avoiding $\ol{31}2$, which is equivalent to avoiding $312$. So consider $\pi \in \S_n(312)$ with $\hat\pi$ having no fixed points. We claim these permutations satisfy the recursion 
    \[
    % a_n = a_{n-1} + 2 a_{n-2} + \sum_{k=3}^{n-1} (a_{k-1} + a_{k-2})a_{n-k}.
    a_n = \sum_{k=2}^{n} (a_{k-1} + a_{k-2})a_{n-k},
    \]
    where $a_0 = 1$. To see this, consider the position of the $1$ in $\pi$. We cannot have $\pi_1 = 1$, this implies $\hat\pi_1 = 1$, which is a fixed point. If $\pi_k = 1$, all the elements to the left of $1$ must be smaller than the elements to its right to avoid $312$. This means $\pi = \tau 1\sigma$, with $\tau < \sigma$. The element to the right of $1$ is therefore a left-to-right maximum, and so begins a new cycle in $\hat\pi$. 
    
    In order for $\hat\pi$ to avoid any fixed points, ${\rm red}(\sigma)$ must also avoid having fixed points. This means ${\rm red}(\sigma) \in \S_{n-k}(312; (1; 1\to1))$ and so is enumerated by $a_{n-k}$, where in the case $n = k$, we let $a_0 = 1$. 
    
    Similarly, ${\rm red}(\tau)$ could either have no fixed points, or could have its final point fixed because the $1$ would extend this cycle and get rid of the fixed point in $\pi$. This, however, is the only fixed point ${\rm red}(\tau)$ can have, otherwise $\sigma$ has a fixed point. If ${\rm red}(\tau)$ has no fixed points, there are $a_{k-1}$ options. If it has one fixed point at the end, there are $a_{k-2}$ options. In the case $k = 2$, there is one option for $\tau$, namely ${\rm red}(\tau) = 1$, this exceptional case where $\pi$ begins with $21$ gives $(a_{1} + a_0)a_{n-2} = a_{n-2}$ total options. Summing over all positions $k$ yields the stated recursion.

    Solving for the generating function, we find $A(x) = \frac{1 + x - \sqrt{1 - 2x - 3x^2}}{2x(1+x)}$, which is precisely the generating function for the Riordan numbers.
\end{proof}

\begin{remark}
    In \cite{CDY08}, the authors' motivation was to find a combinatorial interpretation of the Riordan numbers in terms of avoidance in permutations. This gives a new characterization of the Riordan numbers in terms of arrow avoidance. To see the relation to the Motzkin numbers, we note that $M_n = r_{n+1} + r_n$, where $r_n$ is the Riordan number. Riordan numbers are typically expressed by the convolution $r_n = \sum_{k=2}^n M_{k-2}r_{n-k}$. Substituting $M_{k-2} = r_{k-1} + r_{k-2}$ recovers our recursion. 
\end{remark}

\begin{theorem} \label{thm:13t23t11}
    For $n \geq 2$, $a_n((12;2 \to 3), (1;1\to1)) = G_{n-2}$, the $(n-2)$-nd Gould Number.
\end{theorem}

\begin{proof}
    For the remainder of this proof, we let $a_n := a_n((12;2\to3), (1;1\to1))$. Following the proof of Theorem \ref{thm:12t23}, since $\pi \in \S_n(12;2\to3)$ we must have $\pi_n = n$ or $\pi_n = 1$. However, since $\hat\pi$ also avoids fixed points, we must have $\pi_n = 1$. We now, again, consider the two possible cases: there is no ascent between $n$ and $1$, or there is some ascent between $n$ and $1$ in $\pi$. In the first case, as before, there are $\binom{n-2}{k} a_{n-k-2}$ options. We choose the elements that occur between $n$ and $1$, the other elements are arranged in any way that still avoids $(12;2\to3)$ and $(1;1\to1)$.

    In the second case, we again get an identical sum as in the proof of Theorem \ref{thm:12t23}. However, now since all permutations in $\S_n((12;2\to3),(1;1\to1))$ necessarily have $\pi_n = 1$, we find there are $\sum_{i=2}^{n-2} \sum_{k=1}^{n-i-1} \binom{n-i-1}{k} a_{n-i-k+1}$ such permutations. Indeed, we are again enumerating the $\pi$ with final ascent between $n$ and $1$ beginning with $i \in \{2,\dots,n-2\}$. We can again choose the $k$ elements between $i$ and $i-1$, but now the remaining elements after removing the entries to the right of $i$ are enumerated by $a_{n-i-k+1}$.  In all we conclude,
    \[
    a_n = \sum_{k=0}^{n-2} \binom{n-2}{k} a_{n-k-2} + \sum_{i=2}^{n-2} \sum_{k=1}^{n-i-1} \binom{n-i-1}{k} a_{n-i-k+1}
    \]
    This simplifies to,
    \[
    a_n = \sum_{k=0}^{n-2} \binom{n-2}{k} a_{k+1}.
    \]
\end{proof}

As with Theorem \ref{thm:12t23} it would be interesting to find a simpler construction that produces the same recursion.

\begin{theorem} \label{thm:13t32t11}
    For $n \geq 2$, $a_n := a_n((12;3\to2),(1;1\to1))$ satisfies
    \[
    a_n = \sum_{k=2}^{n} \binom{n-2}{k-2}(a_{k-1}+a_{k-2}) = \sum_{k=0}^{n-1} \binom{n-1}{k} a_k 
    \]
    with $a_0 = 1$ and $a_1 = 0$.
\end{theorem}

\begin{proof}
    First, by Lemma \ref{lem: vincArrow}, avoiding $(12;3\to2)$ is equivalent to avoiding $1 \ol{32}$. So, to count $\pi \in \S_n(1 \ol{32},(1;1\to1))$, we first note that $n$ cannot appear to the right of $1$. This is because all entries to the right of $1$ must be increasing to avoid $1 \ol{32}$, so if $n$ appears to the right of $1$ this means $\pi_n = n$, which means $\hat\pi_n = n$, which is not allowed.

    Now, consider the permutations with $\pi_k = 1$. As noted above, we must have $\pi_{k+1} < \pi_{k+2} < \cdots < \pi_n$ to avoid a $1 \ol{32}$. Consider the permutation formed by the entries to the left of $1$, $\pi'  = {\rm red}(\pi_1\pi_2\cdots\pi_{k-1})$. Then $\pi'$ must still avoid $1 \ol{32}$ and either avoid fixed points, or have $\pi_{k-1} = k-1$, since $1$ will remove this fixed point in $\pi$. This means there are $a_{k-1} + a_{k-2}$ such permutations. So to construct $\pi$ with $\pi_k = 1$, we first choose the first $k-2$ entries other than $n$, then arrange those entries in $a_{k-1} + a_{k-2}$ ways. The entries after $1$ must be arranged in increasing order. Summing over all allowable positions where $1$ could occur gives:
    \[
    a_n = \sum_{k=2}^{n} \binom{n-2}{k-2}(a_{k-1} + a_{k-2}).
    \]
    Simplifying this sum results in the second stated recursion.
\end{proof}

\begin{remark}
    We note that this recursion is identical to the recursion satisfied by the Bell numbers, the difference being that $a_1((12;3\to2),(1;1\to1)) = 0$, whereas $B_1 = 1$. This also implies that if you consider $B_n - a_n((12;3\to2),(1;1\to1))$, the $1\ol{32}$ avoiding permutations that contain a fixed point, they also satisfy the same recursion.
\end{remark}

\section{Conjectures and Open Questions} \label{sec:conjs}

There are many possible future directions of study. In this paper, we've enumerated $\S_n(\alpha)$ for all $\alpha=(\nu;H)\in \A_k$ with $k\in\{1,2,3\}$, $|\nu|\in\{1,2\}$, and $|H|=1$, except for the cases when $\alpha=(12;3\to3)$ or $\alpha=(21;3\to3)$, which remain open. We've also enumerated $\S_n(\alpha,\beta)$ where $\alpha=(12;H)\in\A_2$ or $\A_3$ with $|H|=1$ and $\beta=(1;1\to1).$ There are clearly still many cases that have not yet been considered, but which may have interesting results. In particular, any case where $|\nu|>2$ or where $|H|>2$ are not addressed in this paper.  

Additionally, arrow pattern avoidance can be paired with classical pattern avoidance. In these cases, some interesting sequences seem to appear. For example, we make the following conjecture. 

\begin{conjecture}
    For $n\geq 2,$
    \begin{itemize}
        \item $a_n(123, (12; 1 \to 3)) = 2^{n}-n$,
        \item $a_n(321,(12; 1 \to 3)) = F_{2n-1}$, and 
        \item  $a_n(321,(12;1 \to 2)) = M_n$,
    \end{itemize}
     where $F_n$ is the $n$-th Fibonacci number and $M_n$ is the $n$-th Motzkin number.
\end{conjecture}

In Section~\ref{sec:arrows}, we define the notion of arrow-Wilf equivalence. As seen in this paper we have the following arrow patterns in the same arrow-Wilf equivalence class: 
\[
\{(12;1\to2), (12;3\to2), (21;3\to2),(13;1\to2), (13;2\to1),(13;3\to2)\}. 
\]
None of these are explained by Proposition~\ref{prop: rc1}. It would be interesting to have more theorems that explain certain general arrow-Wilf equivalences, if such theorems exist.

Finally, as stated in the introduction, certain sets of permutations can be characterized in terms of arrow pattern avoidance, such as shallow permutations or 321-avoiding cyclic permutations. It seems likely that other sets of permutations in which both the one-line and algebraic structure of the permutation are utilized can be described in a similar fashion. Developing techniques for enumerating permutations avoiding arrow permutations could help us develop a better understanding of these sets of permutations, and perhaps enumerate 321-avoiding cyclic permutations, which are currently not well-understood.

\bibliographystyle{amsplain}

\noindent {\tiny \emph{The views expressed in this paper are those of the authors and do not reflect the official policy or position of the U.S. Naval Academy, Department of the Navy, the Department of Defense, or the U.S. Government.} \par}

\end{document}